\documentclass[a4paper,12pt,reqno]{amsart}

\usepackage{graphicx,pdfsync,amssymb, latexsym,amsfonts,amsbsy,color,subcaption,esint,mathtools,enumerate,nicefrac,tikz}
\usepackage{hyperref}
\hypersetup{
colorlinks=true,
linkcolor=blue,
filecolor=blue,
urlcolor=blue,
citecolor=blue,
}

\usepackage{float}
\restylefloat{table}

\usepackage[utf8]{inputenc}% for special characters
\usepackage[T1]{fontenc}% for special characters

\usepackage[top=1.1in, bottom=1in, left=1in, right=1in]{geometry}

\DeclareMathOperator{\dist}{dist}

\newtheorem{theorem}{Theorem}[section]
\newtheorem{lemma}[theorem]{Lemma}
\newtheorem{proposition}[theorem]{Proposition}
\newtheorem{definition}[theorem]{Definition}

\newtheorem{assumption}[theorem]{Assumption}

\usepackage{times}
\usepackage{setspace}
\def \RR {\mathbb{R}}  %real numbers
\def \NN {\mathbb{N}}  %natural numbers
\def\cC {\mathcal C}

\def \o {\omega}

\def \p {\partial}

\def \ep {\varepsilon}

\numberwithin{equation}{section}

\begin{document}

\title{On nonexistence of splash singularities for the $\alpha$-SQG patches}

\author{Alexander Kiselev}
\address{Department of Mathematics, Duke University, Durham, NC 27708}
\email{kiselev@math.duke.edu}

%\author{Xiaoyutao Luo$^{\ast}$}
%\address{Department of Mathematics, Duke University, Durham, NC 27708}
%\email{xiaoyutao.luo@duke.edu}
%\thanks{$^{\ast}$Corresponding author.}

\author{Xiaoyutao Luo}
\address{Department of Mathematics, Duke University, Durham, NC 27708}
\email{xiaoyutao.luo@duke.edu}

\date{\today}

\begin{abstract}
In this paper, we consider patch solutions to the $\alpha$-SQG equation and derive new criteria for the absence of splash singularity
where different patches or parts of the same patch collide in finite time. Our criterion
refines a result due to Gancedo and Strain \cite{GS}, providing a condition on the growth of curvature of the patch necessary
for the splash and an exponential in time lower bound on the distance between patches with bounded curvature.
\end{abstract}
\maketitle

%%%%%%%%%%%%%%%%%%%%%%%%%%%%%%%%%%%%%%%%%%%%%%%%%%%%%%%%%%%%%%%%%%%%%%%%%%%%%%%%%%
\section{Introduction}
%%%%%%%%%%%%%%%%%%%%%%%%%%%%%%%%%%%%%%%%%%%%%%%%%%%%%%%%%%%%%%%%%%%%%%%%%%%%%%%%%%
Recall that the family of $\alpha$-SQG equations is given by
\begin{equation}\label{eq:alpha_SQG}
\begin{cases}
\p \omega + u \cdot \nabla  \omega =0,&\\
u = \nabla^{\perp}(- \Delta )^{-1+\alpha } \omega.
\end{cases}
\end{equation}
where $ \nabla^{\perp} = (\p_{x_2} , - \p_{x_1})$ denotes the perpendicular gradient.
The value $\alpha=0$ in \eqref{eq:alpha_SQG} corresponds to the Euler equation, and
$\alpha = 1/2$ to the SQG equation. In general, models with $\alpha$ in the range
$(0,1)$ have been considered \cite{CIW,CCCGW}.
The $\alpha$-SQG equations appear in atmospheric and ocean science (see \cite{CMT,Held}), and model
evolution of temperature near the surface of a planet.
Mathematically, the SQG equation has some similarities with the 3D Euler equation \cite{CMT},
and has been a focus of much attention in recent years. The global regularity vs finite time blow up
question for smooth initial data remains open for any $1>\alpha>0.$
 A singular scenario, closing front, has been presented in \cite{CMT}. However, later rigorous work
\cite{Cord,CordF,CFdL} has shown that finite time blow up cannot happen in this scenario.

The SQG equation is in particular used to model frontogenesis: an interface with a sharp jump
of temperature across it. In this context, patch solutions are natural. These are weak solutions
of the equation that have form
\[ \omega(x,t) = \sum_{j=1}^n \theta_j \chi_{\Omega_j(t)}(x), \]
where $\theta_j$ are constants, $\chi_S$ denotes the characteristic function of the set $S$,
and $\Omega_j(t)$ are disjoint, regular regions evolving in time according to the Biot-Savart law \eqref{eq:alpha_SQG}
(this will be made more precise later). In the context of patches, the global regularity question is whether the
patch solution conserves the initial regularity class of the boundaries $\partial \Omega(0),$ and whether
different patches can collide or self-intersect.
The existence and uniqueness of patch solutions for the 2D Euler
equation is a consequence of Yudovich theory (see \cite{Yudth,MP}), and global regularity
has been proved by Chemin \cite{c}. For $\alpha>0,$ even the existence of patch solutions is not trivial.
Local existence and uniqueness results of $\alpha$-SQG patch solutions have been
proved in \cite{Rodrigo,g,CCG,CCCGW}.
Numerical simulation in \cite{CFMR} indicated a possible splash singularity where two patches touch each other with simultaneous
formation of corners at the touch point, yet rigorous understanding of the phenomenon remained missing.
For small $\alpha>0,$ finite time singularity
formation has been proved for patches in the half-plane setting \cite{KRYZ}.
This singularity formation happens near the hyperbolic point of the flow on the boundary, and in a scenario similar to very fast small scale
growth in solutions to 2D Euler equation \cite{KS,KL} and conjectured singularity formation in the 3D Euler
Hou-Luo scenario \cite{HL}. On the other hand, there are also recent numerical simulations by Scott and Dritschel \cite{SD1,SD2}
which suggests a different pathway towards a singularity. In \cite{SD1}, an intricate self-similar cascade of filament instabilities
is explored, where the picture roughly repeats in different locations at a geometric sequence of decreasing length scales and time intervals. In \cite{SD2},
it is suggested that filament pinching may happen in a simpler fashion, at one of the stages of the previous instability cascade.
This filament pinching might be of the type of splash singularity, where different parts of the patch boundary touch each other.
The formation of a splash singularity has been rigorously established for water waves \cite{CCFGG,CS}, but the difference with $\alpha$-patch case
is that the wave interfaces can remain regular near the intersection. This cannot happen for the SQG patches, at least not in a simple way,
as the parts of the patch(es) with bounded
curvature will never collide as shown by Gancedo and Strain \cite{GS}. Thus a simple $\alpha$-SQG splash singularity can only happen along with
the loss of regularity of the patch boundary, and rigorous examples of it remain missing possibly apart from \cite{KRYZ} (where a boundary is
present and the precise picture of singularity is not established).

In this note, our goal is to sharpen the criterion of ruling out splash singularity for the $\alpha$-SQG models, in order to understand what
kind of phenomena - specifically, the rate of growth of certain norms associated with patch regularity -  must appear for the splash to occur.
We also prove a sharper separation result for patches with bounded curvature,
improving the double exponential bound of \cite{GS} to just exponential in time. The key observation for these results is an
additional folding odd symmetry in the Biot-Savart law.

\subsection{Criteria of no splash singularities}

In the context of this paper, we will interpret the splash singularity as in the following definition; we will recall the notion of patch solution
more precisely below in Section \ref{ps}.
\begin{definition}\label{defss}
Let $k\in \NN$ and $\gamma \in [0,1]$. We say that a $C^{k,\gamma}$-patch solution $\Omega(t) = \cup_{j=1}^N \Omega_j(t)$ on $[0,T)$ develops a splash singularity as $t \to T^-$ if there exists $\ep>0$ and a fixed ball $B_\rho(x_0),$ $\rho>0 $ such that on $[T-\ep, T)$ all of the following holds.
\begin{itemize}
    \item There are only two disjoint branches $\cC_1(t)$ and $\cC_2(t)$ of the interface in the ball $B_\rho(x_0)$ that are simple curves, %with uniformly bounded $C^{k,\gamma}$ norms
and $\cC_1(t)$ and $\cC_2(t)$ touch at a single point $x_0$ as $t\to T^-$.

\item  In the complement of the ball $B_\rho(x_0)$, the patch solution remains regular in the whole time interval  $[0,T]$: the $C^{k,\gamma}$ norms of the patches remain
bounded, different patches do not touch and individual patches do not self-intersect.

\item The $C^{k,\gamma}$ regularity may be lost at time $T$ as interfaces develop singular structures,
but the singularity is localized at $x_0.$
\end{itemize}

\end{definition}
For the sake of simplicity, we can think of all norms defined in terms of intrinsic  arc length parametrization, with
\[ \|\p \Omega(t)\|_{C^{k,\gamma}} = {\rm max}_{j}\left( \sum_{l=0}^k \|\partial_s^l x_j(t)\|_\infty + {\rm sup}_{s,r} \frac{|\partial_s^k x_j(s)- \partial_s^k x_j(r)|}{|s-r|^\gamma} \right),\]
where $x_i$ is the arc length parametrization of the $i$th patch $\Omega_i(t)$ and $s, r$ are arc length parameters.
The lack of self-intersection outside $B_\rho(x_0)$ can be rigorously
defined as a positive lower bound for the arc chord ratio: \[ {\rm min}_{j,x_j(s),x_j(r)\notin B_\rho} \frac{|x_j(s)-x_j(r)|}{|s-r|} \geq c>0. \]
%Note that the $C^{k,\gamma}$ regularity may be lost at time $T$ as interfaces develop singular structures.

\begin{theorem}\label{thm:main_thm_no_splash}
Let $0 < \alpha < 1$, $k :=    \lceil 2 \alpha \rceil$ and $  1 \geq \gamma  \geq   2 \alpha  - \lceil 2 \alpha \rceil +1.$ Suppose that $\o$ is a $C^{k,\gamma}$-patch solution to \eqref{eq:alpha_SQG} on $[0,T)$
that forms a splash singularity at time $T.$ Then we must have
\begin{equation}\label{m1} \int_0^T \|\p \Omega(t)\|_{C^{k,\gamma}}^{\frac{ k+ 2\alpha -1   }{k+\gamma- 1 }}\,dt = \infty. \end{equation}
%If $\alpha =1/2,$ we must have
%\begin{equation}\label{m2} \int_0^T \|\p \Omega\|_{C^{1,1}}^{2\alpha}\,dt = \infty. \end{equation}
%Finally, if $1/2 < \alpha <1,$ we must have
%\begin{equation}\label{m3a} \int_0^T \|\p \Omega\|_{C^{1,1}}^{2\alpha}\,dt = \infty \end{equation}
%or
%\begin{equation}\label{m3b} \int_0^T \|\p \Omega\|_{C^{2,\gamma}}\,dt = \infty. \end{equation}
%If there exists  $\ep>0$ such that
%$$
%\| \p \Omega(t) \|_{C^{k, \gamma }} \lesssim (T- t)^{ - \frac{k+ \gamma}{2\alpha } +\ep }
%$$
%then $w(x,t)$ does not develop any splash singularity as $t \to T^-$.
\end{theorem}
In particular, the curvature of $\p \Omega$ controls splash singularity for all $\alpha \leq 1/2$ (as was shown in \cite{GS} for $\alpha=1/2$).
In fact, for $\alpha=1/2$ case we can replace $C^2$ norm of \cite{GS}  by $C^{1,1}$ norm  in \eqref{m1} (functions with Lipschitz derivative).

\subsection{Exponential bound of minimal distance}

In the case where we have a priori control on growth of the appropriate norms of patch solution, we can derive a lower bound on separation distance between different parts of the patch boundary provided that at any time the minimal distance is achieved at two points with certain properties. In particular, we need to make an assumption limiting the nature of how the patch boundary can approach itself.
\begin{assumption}\label{as1}
Assume that the $C^{k,\gamma}$ patch solution $\Omega(t)$ satisfies the following property: there exists $\eta>0$ and $c>0$ such that for all $t \in [0,T)$ and all $i=1,\dots, n,$ $|x_i(s)-x_i(r)| \geq c|r-s|$ for all $s,r$ with $|s-r| \leq \eta.$
Here $x_i$ is the arc length parametrization of the $i$th patch $\Omega_i$.
\end{assumption}

We explain all the details later in the paper, but let us state here the main result (referring to forthcoming definitions).

\begin{theorem}\label{thm:main_thm_exp_bound}
Let $0 < \alpha < 1$,  $k:=   \lceil 2 \alpha \rceil$, and $1 \geq \gamma  \geq   2 \alpha  - \lceil 2 \alpha \rceil +1 $. Suppose that $\o$ is a $C^{k,\gamma}$-patch solution to \eqref{eq:alpha_SQG} on $[0,T )$
satisfying Assumption \ref{as1}. Let $d(t)$ be the separation distance defined in \eqref{ddef}. Assume in addition that for a.e. time $t \in [0,T],$ all points $p$, $q$ such as $|p-q|=d(t)$ are admissible in the sense of Definition \ref{def:admissible}.
Then for all $t \in [0,T],$ we have
\[ d(t) \geq d(0) \exp  \left(-C\int_0^t \|\p \Omega(t' )\|_{C^{k,\gamma}}^{\frac{k+ 2\alpha -1   }{k + \gamma -1  }} \,d t' \right), \]
$C=C(\alpha,k,\gamma, \eta).$
%Here
%\[ D(t) = \left\{ \begin{array}{ll} \|\p \Omega(t)\|_{C^{1,\gamma}}^{2\alpha/\gamma} & 0 < \alpha <1/2 \\
%\|\p \Omega\|_{C^{1,1}}^{2\alpha} & \alpha =1/2 \\  \|\p \Omega\|_{C^{1,1}}^{2\alpha}+\|\p \Omega\|_{C^{2,\gamma}} & 1/2 <\alpha <1.
% \end{array} \right. \]
%If there exists $M>0$ such that
%$$
%\sup_t \| \p \Omega(t) \|_{C^{1+k,\gamma }} \leq M,
%$$
%and two branches of the interface are approaching as $t\to T^-$, then the minimum distance between these two branches is bounded from below by $e^{-c t}$ where the constant $c>0$ depends on $\alpha$, $\gamma$, the initial distance, and initial values of the patches.
\end{theorem}

Note that if the curvature of $\p \Omega(t)$ is bounded uniformly in time, we obtain an exponential in time lower bound on how quickly the patch boundaries can approach for $0 < \alpha \leq 1/2.$
%As before, for $\alpha=1/2$ we can use $C^{1,1}$ instead of $C^2.$

We remark that as this paper was being completed, Andrej Zlatos \cite{AZ} has told us that jointly with Junekey Jeon they are able to show absence of patch singularities for $\alpha$-patches with bounded $C^{1,2\alpha}$ norm
without additional assumptions on the geometry of the solution for $0 \leq \alpha < 1/4.$

%\subsection{Exponential bound for $\alpha$-SQG front}

%%%%%%%%%%%%%%%%%%%%%%%%%%%%%%%%%%%%%%%%%%%%%%%%%%%%%%%%%%%%%%%%%%%%%%%%%%%%%%%%%%
\section{Preliminaries}
%%%%%%%%%%%%%%%%%%%%%%%%%%%%%%%%%%%%%%%%%%%%%%%%%%%%%%%%%%%%%%%%%%%%%%%%%%%%%%%%%%

%%%%%%%%%%%%%%%%%%%%%%%%%%%%%%%
%\subsection{The $\alpha$-SQG equation}
%%%%%%%%%%%%%%%%%%%%%%%%%%%%%%%

%%%%%%%%%%%%%%%%%%%%%%%%%%%%%%%
\subsection{Definition of a patch solution}\label{ps}
%%%%%%%%%%%%%%%%%%%%%%%%%%%%%%%
The explicit form of the Biot-Savart law for the $\alpha$-SQG equation \eqref{eq:alpha_SQG} that is valid for smooth $\omega$  is given by
\begin{equation}\label{eq:biot-savart}
u(x, t) = P.V. \int_{\RR^2} \frac{(x- y)^\perp }{|x - y|^{2 + 2\alpha }} \o(y,t) \, dy
\end{equation}
(we omit the constant $c(\alpha)$ in front of the integral).
For patch solutions and $\alpha \geq 1/2,$ the tangential to patch component of the velocity is infinite even for smooth patches, so we will
only deal with normal component defining the patch evolution.
%To define the patch solution, let us recall that Hausdorff distance
We adapt the definition of patch solution similar to \cite{KRYZ}. Recall that the Hausdorff distance between any two sets $A$ and $B$ is given by
\[ d_H(A,B)= {\rm max} ({\rm sup}_{a \in A} {\rm dist}(a,B), \,\, {\rm sup}_{b \in B} {\rm dist}(A,b)). \]
\begin{definition}\label{splash1016}
Let $T>0$, $N<\infty$ be an integer, and $\theta_i \in \RR$ for $1\leq i \leq N$. Suppose that $\Omega_i (t) \subset \RR^2$ are bounded open sets whose boundaries $ \p \Omega_i (t)$ are pairwise disjoint closed $C^1$ curves for every $t\in [0 ,T) $. Let $\Omega = \bigcup_{1\leq i \leq N} \Omega_i$ and
$$
\omega( \cdot ,t) : = \sum_{1\leq i\leq N} \theta_i \chi_{\Omega_i(t)}.
$$

 We say $\omega$ (or just $\Omega$) is a patch solution for the $\alpha$-SQG equation on $[0,T )$ if the followings are satisfied.
\begin{enumerate}
\item Each $ \p \Omega_i$ is continuous in $t$ with respect to the Hausdorff distance $d_H$.

\item Denote $\p \Omega (t) : = \cup_i \p \Omega_i (t).$ Then for every $t$ while patch solution exists
\begin{equation}
\lim_{h \to 0+} \frac{d_H \left (\p \Omega (t+h), X_{u_n(\cdot, t)}^h[\p \Omega (t ) ]    \right)}{h} =0
\end{equation}
where $X_{v}^h[E ]: =\{x + h v(x): \, v\in E \} $ and $u_n = (u \cdot n)n$ is the normal to the boundary $\partial \Omega(\cdot, t)$ component of the
velocity field given by \eqref{eq:biot-savart}.
\end{enumerate}
\end{definition}
The tangential component of the velocity \eqref{eq:biot-savart} is infinite for any regularity of the boundary starting from $\alpha=1/2,$
and this explains our difference with the definition of \cite{KRYZ} where only small values of $\alpha$ were considered.

The following elementary lemma shows that $u_n = (u \cdot n)n$ is well-defined for patches with even less regularity than our case; %can be proved similarly to our arguments below, and we ill omit its proof:
in the above formula, we first take the inner product with the normal at $x \in \p \Omega$ and then integrate according to \eqref{eq:biot-savart}.
In the next lemma, we abuse notation and denote by $\Omega$ a single fixed $C^{1,\gamma}$ domain (not necessarily a patch solution).

\begin{lemma}\label{uinf}
Let $0 < \alpha <1,$ and suppose
that $\Omega$ is a  compact, connected $C^{1,\gamma}$ domain with $\gamma  >  {\rm max}(0,2 \alpha  - 1)$.
Then $u_n = (u \cdot n)n$, the normal to $\p \Omega$ component of the velocity $u$, computed according to \eqref{eq:biot-savart} with $\omega(x) =\chi_\Omega(x)$ is well defined and finite at all points on $\p \Omega,$ and is continuous on $\p \Omega$ with
\begin{equation}\label{ubd}
\|u_n\|_{C(\partial \Omega)} \leq C(\Omega).
\end{equation}
\end{lemma}
\begin{proof}
Fix a point on $\p \Omega,$ and choose a system of coordinates with this point at its origin such that $x_1$ axis is along the tangent to $\p \Omega$.
Then
\begin{equation}\label{uexp} u_n \cdot n = -P.V. \int_{\Omega} \frac{y_1}{|y|^{2+2\alpha}} \,d y. \end{equation}
Consider a square $S_R=[-R,R]^2$ with $R= 0.1 \|\p \Omega\|_{C^{1,\gamma}}^{-1}$ centered at the origin, with one of its sides parallel to $y_1$ axis.
Let $f(y_1)$ be a function such that
the graph $y_2= f(y_1)$ within $S_R$ coincides with the part of $\p \Omega$ containing the origin. If $S_R$ contains more than the curve of $\p \Omega$
passing through the origin, make $R$ smaller so that there is only one curve of $\p \Omega$ in it.
The part in \eqref{uexp} coming from integration over
$S_R^c$ can be estimated from above by
\begin{equation}\label{linf1} \int_{|y| \geq R} |y|^{-1-2\alpha} \chi_{\Omega}(y)\,dy \leq C(\Omega) {\rm max}(1,R^{1-2\alpha}) \end{equation}
due to the compactness of $\Omega.$
To exploit the odd in $y_1$ symmetry, for any $y_1 \geq 0$ we introduce
\begin{equation}\label{ulf}
\overline{f}(y):= \max\{f(y), f(-y)) \} \quad \text{and }\quad \underline{f}(y):= \min\{f(y), f(-y)) \}.
\end{equation}
Then, with \eqref{ulf} we have
\begin{equation}\label{aux1017} \left| \int_{\Omega \cap S_{R}} \frac{y_1}{|y|^{2+2\alpha}} d \,y \right| \leq \int_0^{R} \int_{\underline{f}(y_1)}^{\bar{f}(y_1)} \frac{y_1}{|y|^{2+2\alpha}}\,dy_2dy_1. \end{equation}
To facilitate the estimates let us note that since $f$ is $C^{1, \gamma}$ and $f(0)=f' (0)= 0$, we have
$$
| \overline{f}(y_1) -    \underline{f}(y_1)| \leq \|\p \Omega   \|_{C^{1,\gamma} }  |y_1|^{1+\gamma } \quad \text{for all $ |y_1| \leq R$}.
$$
Also, in the region of integration on the right of \eqref{aux1017} $y_2 \leq y_1$ due to the regularity of $\p \Omega$ and choice of $R$.
Then we can continue the estimate \eqref{aux1017} and obtain upper bound by
\begin{equation}\label{Rbound} C(\Omega) \int_0^{R} y_1^{2+\gamma} y_1^{-2-2\alpha} dy_1 \leq C(\Omega)R^{1+\gamma-2\alpha}. \end{equation}
Given that $R$ only depends on $\Omega,$ the bounds \eqref{linf1} and \eqref{Rbound} imply \eqref{ubd} for the $L^\infty$ norm of $u_n.$
The continuity of $u_n$ on $\p \Omega$ follows by Vitali convergence theorem. Indeed, let us perform
an odd reflection like above in the integrals defining $u_n(x)$ and $u_n(x').$ Then the regions of integration and integrands converge pointwise
as $x' \rightarrow x$ along $\p \Omega$. Also, the integrands are uniformly integrable due to elementary estimates using the structure of integration
regions similarly to \eqref{aux1017}, \eqref{Rbound}.
\end{proof}
%Note that an elementary computation (that in particular follows from our argument below) shows that the normal component of the velocity
%is everywhere defined and bounded if $\partial \Omega \in C^{k, \gamma}$ with $k =    \lceil 2 \alpha \rceil$ and $\gamma  >   2 \alpha  - k$.
%Lower regularity (on the scale of H\"older spaces) in general does not guarantee pointwise existence of $u_n$ if $\alpha \geq 1/2.$

In this note, we do not discuss questions of the existence and uniqueness of patch solutions. Existence of patch solutions with $C^\infty$ or Sobolev regularity
follows from the contour equation analysis in \cite{g,Rodrigo} for $0<\alpha \leq 1/2,$ and in \cite{CCCGW} for $\alpha > 1/2.$
The uniqueness of patch solutions in the whole space setting is known for $\alpha \leq 1/2$ \cite{g,CCG} and is open to the best of our knowledge in the case $\alpha>1/2.$

%%%%%%%%%%%%%%%%%%%%%%%%%%%%%%%%%%%%%%%%%%%%%%%%%%%%%%%%%%%%%%%%%%%%%%%%%%%%%%%%%%
\section{Absence of splash singularities}
%%%%%%%%%%%%%%%%%%%%%%%%%%%%%%%%%%%%%%%%%%%%%%%%%%%%%%%%%%%%%%%%%%%%%%%%%%%%%%%%%%

\subsection{Geometric configuration}

Given our definition of splash singularity, starting from time $t=T-\ep$ there exists a fixed ball $B_\rho(x_0)$ and
a pair of points $p(t),q(t) \in B_\rho$ such that $|p-q|$ is the minimal distance between any two patches (or the maximum of the arc-chord condition if it is different
parts of the same patch that form splash singularity). Furthermore, for each $t \in [T-\ep, T),$ $\partial \Omega(t) \cap B_\rho$  consists of two
disjoint simple $C^{k,\gamma}$ curves one of which contains $p$ and the other $q.$
Since there is only one touching point $x_0$ at time $T$ and the motion of $\p \Omega_j$ is continuous
in time, we can freely assume that any such pair $p,q \in B_{\rho/4}(x_0):$ this will be true starting from some time $t \geq T-\ep.$
%It will be convenient for us to define

%and consider balls $B_r(p)$ and $B_r(q)$ centered at

%If $|p-q| \geq r/2,$ then the patches are still far from splash, this will be our lower
%bound on the separation distance at time $t$. If $|p-q|\leq r/2,$ we claim that $B_r(p)$, similarly to $B_\rho,$ contains exactly two disjoint pieces of $\partial \Omega(t).$
%Indeed, since $p,q \in B_\rho \cap B_r(p)$ the two curves from $B_\rho$ also cross $B_r(p).$ By \eqref{rad} the slope of any curves in $\partial \Omega(t) \cap B_r(p)$
%can only vary by $0.02$ within this ball - and since $r \leq \rho/2$ it is not hard to see that if there were any additional curves in $B_r(p)$ they would have to cross
%$B_\rho,$ too.

\subsection{Estimates for the velocity difference}

We now prove a result on relative velocity of the points $p$ and $q$ inspired by the above discussion. Given $k$ and  $ \gamma,$ define
\begin{equation}\label{rad} r(t) =  \frac14 {\rm min} (\rho, (0.01 \|\partial \Omega(t)\|_{C^{k,\gamma}})^{-\frac{1}{k+\gamma - 1}}).
%\\ \frac14 {\rm min} (\rho, (0.01 \|\partial \Omega(t)\|_{C^{2}})^{-1}), & k=2. \end{array} \right.
\end{equation}
Without loss of generality, we will assume that $\rho \leq 1$ and so $r<1.$

\begin{definition}\label{def:admissible}
Let $0 < \alpha < 1$, $k:=    \lceil 2 \alpha \rceil,$ and $1 \geq \gamma  \geq   2 \alpha  - \lceil 2 \alpha \rceil +1$.  Suppose that $\o$ is a $C^{k,\gamma}$-patch solution  of \eqref{eq:alpha_SQG} and $\Omega (t)$ is the union of the patches.

We say $p,q \in \p \Omega  (t)$ is a pair of admissible points at time $t$ if
\begin{itemize}

\item  $\p \Omega(t) \cap (B_{2r}(p) \cup B_{2r} (q))$ consists of two disjoint curves $\cC_1$ and $\cC_2,$ one containing $p$ and the other $q;$
here $r$ is given by \eqref{rad}.
%there exists a ball are exactly two disjoint $C^{k,\gamma}$ curves $\cC_1,$ $\cC_2$ in $\partial \Omega(t) \cap B_{\rho/2}(p),$ one containing $p$ and the other $q$,
% where the radius $r$ is defined by
%$$
%r: =   100^{-1} \big(  \|\p \Omega   \|_{C^{k,\gamma}   }  \big)^{-\frac{1}{\gamma }};
%$$

\item the distance between $p$ and $q$ forms a local minimum distance, i.e. % there exists $\ep >0$ satisfying
$$
|p- q| = \dist (\cC_1, \cC_2).
%\min_{0< r< \ep } \dist(B_r(p) \cap  \p \Omega,  B_r(q) \cap \p \Omega ).
$$
\end{itemize}
\end{definition}
Note that in general, there could be other patches between $\cC_1$ and $\cC_2$ for example if $|p-q| \gg r$ -
the definition of admissible points does not preclude that.

\begin{figure}[hb]
 \centering
\begin{tikzpicture}[scale=0.8]

% rect_box
\draw  [dash pattern={on 4.5pt off 4.5pt}, color=red!100] (-3,-3) -- (-3,3.5) -- (3,3.5) -- (3,-3) -- cycle
;

\draw [fill={rgb, 255:red, 210; green, 231; blue, 255 }  ,fill opacity= 1 ] (-3,1.0) .. controls (-2,0.9) and (-1, 0.5) .. (0,0.5) .. controls (1,0.5) .. (3,0.7) -- (3,3.5) -- (-3, 3.5) --(-3,1.0);

\draw [fill={rgb, 255:red, 210; green, 231; blue, 255 }  ,fill opacity= 1 ] (-3,-0.5) .. controls (-2,-0.4) and (-1, 0.0) .. (0,0) .. controls (1,0) .. (3,-0.2) -- (3,-3) -- (-3, -3) --(-3,-0.5);

% horizontal axis
\draw[->] (0,0) -- (0,5) ;%node[anchor=north west] {\scriptsize  $y_2$}
% vertical axis
\draw[->] (-4,0) -- (4,0) ;%node[anchor=north east] {\scriptsize $y_1$}

% bottom_curve

\draw[thick,color= blue!50] (-3,-0.5) .. controls (-2,-0.4) and (-1, 0.0) .. (0,0) .. controls (1,0) .. (3,-0.2);

% top_curve
\draw[thick,color= blue!50] (-3,1.0) .. controls (-2,0.9) and (-1, 0.5) .. (0,0.5) .. controls (1,0.5) .. (3,0.7);

% p,q labels
\filldraw
		(0,0) circle (1pt) node[anchor=north] {\scriptsize $p $}
        (0,0.5) circle (1pt) node[anchor=south west] {\scriptsize $q=(0,\delta) $}
          (0,3.5) circle (1pt) node[anchor=south west] {\scriptsize $(0,r +\delta )  $};

% graph labels
\draw 	(-2.8,1.2) node[anchor=west] {\tiny $g(y_1)$} ;
\draw 	(-2.8,-0.7) node[anchor=west] {\tiny $f(y_1)$} ;
\filldraw 	(-3,0)circle (1pt) node[anchor=south east] {\tiny $-r$} ;
\filldraw 	(3,0)circle (1pt) node[anchor=south west] {\tiny $r$} ;

\end{tikzpicture}
\begin{tikzpicture}[scale=0.8]

% rect_box
\draw  [dash pattern={on 4.5pt off 4.5pt}, color=red!100] (-2,-2) -- (-2,5) -- (2,5) -- (2,-2) -- cycle
;

\draw [fill={rgb, 255:red, 210; green, 231; blue, 255 }  ,fill opacity=1 ] (-2,3.5) .. controls (-2,3.4) and (-1, 3.0) .. (0,3) .. controls (1,3 ) .. (2,3.2)  -- (2,5) -- (-2,5)-- (-2,-0.5);

\draw [fill={rgb, 255:red, 210; green, 231; blue, 255 }  ,fill opacity=1 ] (-2,-0.5) .. controls (-2,-0.4) and (-1, 0.0) .. (0,0) .. controls (1,0) .. (2,-0.2) -- (2,-2) -- (-2,-2) --(-2,-0.5);

% horizontal axis
\draw[->] (0,0) -- (0,6) ;%node[anchor=north west] {\scriptsize  $y_2$};
% vertical axis
\draw[->] (-4,0) -- (4,0) ;%node[anchor=north east] {\scriptsize $y_1$};

% bottom_curve

\draw[thick,color= blue!50] (-2,-0.5) .. controls (-2,-0.4) and (-1, 0.0) .. (0,0) .. controls (1,0) .. (2,-0.2);

% top_curve
\draw[thick,color= blue!50] (-2,3.5) .. controls (-2,3.4) and (-1, 3.0) .. (0,3) .. controls (1,3 ) .. (2,3.2);

% p,q labels
\filldraw
		(0,0) circle (1pt) node[anchor=north] {\scriptsize $p $}
        (0,3) circle (1pt) node[anchor=south west] {\scriptsize $q=(0,\delta) $}
          (0,5) circle (1pt) node[anchor=south west] {\scriptsize $(0,r +\delta ) $};

% graph labels
\draw 	(-1.8,3.5) node[anchor=west] {\tiny $g(y_1)$} ;
\draw 	(-1.8,-0.5) node[anchor=west] {\tiny $f(y_1)$} ;
\filldraw 	(-2,0)circle (1pt) node[anchor=south east] {\tiny $-r$} ;
\filldraw 	(2,0)circle (1pt) node[anchor=south west] {\tiny $r$} ;

\end{tikzpicture}
\caption{Illustration of a pair of admissible points $p,q$. The left has $|p- q| \ll r$ while the right has $|p-q| \gg r$.}
\label{fig:splash_scenario}
\end{figure}
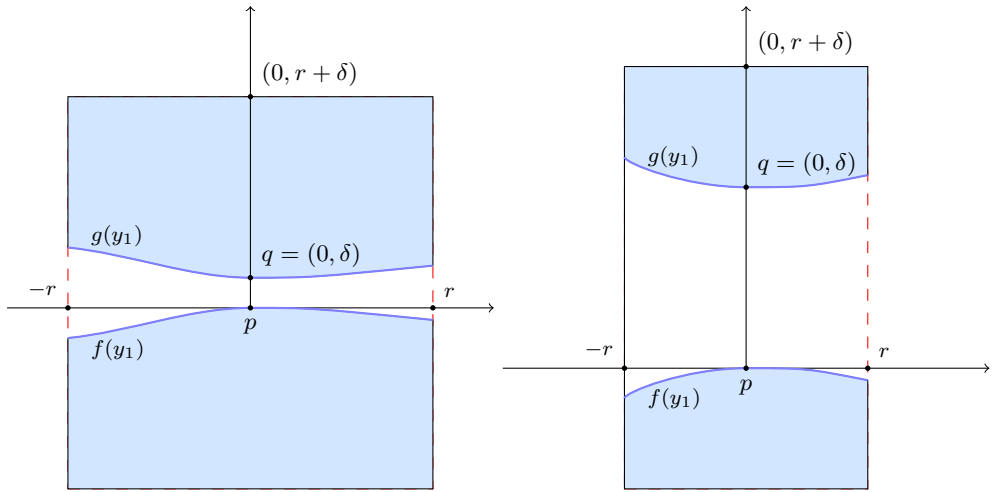

\begin{proposition}\label{prop:main_alpha_all}
Let $0 < \alpha < 1$, $k :=  \lceil 2 \alpha \rceil,$ and $1 \geq \gamma  \geq   2 \alpha  - \lceil 2 \alpha \rceil +1.$
Suppose that $\o$ is a $C^{k,\gamma}$-patch solution and there is a pair of admissible points $p,q \in \p \Omega$ at some time $t$.

%For any $r_0 \lesssim \big(  \|\p \Omega   \|_{C^{1+k,\gamma}   }  \big)^{-\frac{1}{k+\gamma }}$, the following estimates for $\delta := |p-q| $ hold
%Then as $p$ and $q$ are advected with the normal component of the velocity \eqref{eq:biot-savart}, the distance between them satisfies
Then the difference of the normal to patch components $u_n$ of the velocity $u$ defined by \eqref{eq:biot-savart} satisfies
\begin{equation}\label{udiffeq}
\big| u_n(p,t) - u_n(q,t)  \big| \leq C \|\p \Omega(t)\|_{C^{k,\gamma}}^{\frac{k+ 2\alpha -1   }{k+\gamma -1 }}|p-q|.  %\big(  r^{k+ \gamma -2\alpha } \|\p \Omega   \|_{C^{k,\gamma} }  + r^{-2\alpha}\big) |p-q|
\end{equation}
where
%$r$ is given by \eqref{rad} and
the constant $C$ depends  on $\alpha$, $\gamma$, $\rho,$ and the couplings $\theta_j$ of different patches.
%value of each patch, i.e. $\theta_1 , \theta_2 \dots$.
\end{proposition}
% \it Remark. \rm Tracing the argument below, it is easy to check that in the case $\alpha=1/2,$ we can use $C^{1,1}$ norm instead of $C^2.$
\begin{proof}
%For simplicity of  presentation, we only consider the case where $\omega =0$ in the region separating $p$ and $q$. The case of $p$ and $q$ pinching inside of a patch can be handled similarly.
%Let us introduce notation
%\begin{equation}\label{eq:def_r_0}
%r_0= 0.1 r = 0.001  \big(  \|\p \Omega   \|_{C^{1,\gamma}   }  \big)^{-\frac{1}{\gamma }},
%\end{equation}
%so that in a ball of radius $r_0$, the graph of the path boundaries is sufficiently flat.
For simplicity, we drop the time variable $t$ in the proof. Let us set up the coordinate system center at the point $p$ such that the segment of minimum distance is on the $x_2$-axis and $q=(0,\delta)$, $\delta :=|p-q|$. Parametrize the two patch interfaces by
$(x_1, f(x_1) )$ for the bottom piece and $(x_1, g(x_1) )$ for the top piece. It is not difficult to check that with our regularity assumptions on $\p \Omega$ and the choice of $r$ \eqref{rad},
such a representation is valid for $x_1 \in [-r,r].$
%In particular, inside the rectangle $[-r, r] \times [-r, r + \delta]$ these are the only patch boundaries by Definition \ref{def:admissible}
%and choice of $r.$.
See Figure \ref{fig:splash_scenario} for an illustration. Note that when $|q-p|\gg r$, there might be other branches of the patches between $p-q$, but this does not affect our argument.

Denote the coupling constants of the top and bottom patches by $\theta_1$ and $\theta_2$ respectively. The vertical velocity at the points $p$ and $q$ coincides with the normal component of $u$ and is given by
$$
u_n(p) = \theta_1 \int_{E_p} \frac{y_1}{|y|^{2 + 2\alpha }} \, dy + \theta_2 \int_{E_q} \frac{y_1}{|y|^{2 + 2\alpha }} \, dy + \int_{\RR^2 \setminus E  } \frac{y_1}{|y|^{2 + 2\alpha }} \o(y) \, dy
$$
and
$$
u_n(q) = \theta_1 \int_{E_p} \frac{y_1}{|y -(0,\delta(t))|^{2 + 2\alpha }} \, dy + \theta_2 \int_{E_q} \frac{y_1}{|y - (0,\delta(t))|^{2 + 2\alpha }} \, dy + \int_{\RR^2 \setminus E  } \frac{y_1}{|y -(0,\delta(t)) |^{2 + 2\alpha }} \o(y) \, dy
$$
where $E=  E_p  \cup E_q$ and
$$
E_p := \{ x\in \RR^2: - r \leq  x_1   \leq  r,\; \;  -r \leq   x_2 \leq  f(x_1 )     \}
$$
while
$$
E_q := \{ x\in \RR^2:  - r \leq  x_1\leq  r  ,\;  \;  g(x_1 )\leq x_2 \leq  r +\delta   \}
$$

The difference of normal velocities  reads
\begin{align*}
%\frac{d}{dt} \delta (t) =
u_n(q) - u_n(p)  & =  \theta_1 \int_{E_p}y_1 \left( \frac{1}{|y|^{2 + 2\alpha }} - \frac{ 1}{|y - (0,\delta )|^{2 + 2\alpha }}  \right)  \, dy + \theta_2 \int_{E_q}y_1 \left( \frac{1}{|y|^{2 + 2\alpha }} - \frac{ 1}{|y - (0,\delta)|^{2 + 2\alpha }}  \right)  \\
& \qquad\qquad +   \int_{\RR^2 \setminus E  } y_1 \left( \frac{1}{|y|^{2 + 2\alpha }} - \frac{ 1}{|y - (0,\delta)|^{2 + 2\alpha }}  \right) \o(y) \, dy\\
& := I_1 +I_2 +I_3.
\end{align*}

\noindent
{\bf Estimates of $I_1$}

We first split the integral by $y_2$ axis
\begin{equation}\label{eq:I_1}
\begin{aligned}
I_1 &= \theta_1 \int_{0 \leq y_1 \leq r } \int_{ -r \leq y_2 \leq  f(y_1)} y_1 \left( \frac{1}{|y|^{2 + 2\alpha }} - \frac{ 1}{|y - (0,\delta )|^{2 + 2\alpha }}  \right)  \, dy\\
& \qquad -  \theta_1 \int_{0 \leq y_1 \leq r  } \int_{ -r \leq y_2 \leq  f(-y_1)} y_1 \left( \frac{1}{|y|^{2 + 2\alpha }} - \frac{ 1}{|y - (0,\delta )|^{2 + 2\alpha }}  \right)  \, dy
\end{aligned}
\end{equation}
where in the second integral we have applied a change of variable $y_1 \mapsto -y_1$.

To exploit the odd-in-$y_1$ symmetry, for any $y_1 \geq 0$ we introduce
$$
\overline{f}(y):= \max\{f(y), f(-y)) \} \quad \text{and }\quad \underline{f}(y):= \min\{f(y), f(-y)) \}.
$$
With these, due to the opposite signs in \eqref{eq:I_1}, we have
\begin{align}\label{aux10151}
\big|  I_1 \big| \leq  \theta_1  \int_{     0\leq  y_1  \leq r} \int_{\underline{f}(y_1) \leq y_2 \leq \overline{f}(y_1)  } y_1 \left| \frac{1}{|y|^{2 + 2\alpha }} - \frac{ 1}{|y - (0,\delta)|^{2 + 2\alpha }}  \right|  \, dy .
\end{align}

To facilitate the estimates let us note that since $f(0)=f' (0)= 0$, when $f$ is $C^{1, \gamma}$, we have
$$
| \overline{f}(y_1) -    \underline{f}(y_1)| \lesssim \|\p \Omega   \|_{C^{1,\gamma} }  |y_1|^{1+\gamma }
$$
and when $f$ is $C^{2, \gamma}$, by the Fundamental Theorem of Calculus we have
$$
| \overline{f}(y_1) -    \underline{f}(y_1)| = \left| \int_0^{y_1} (f'( x ) + f'(-x)) dx \right| = \left| \int_0^{y_1}\int_0^{x}  ( f''( z ) - f''(-z)) \,dz dx \right|  \lesssim \|\p \Omega   \|_{C^{2,\gamma} }   |  y_1 |^{2+\gamma}.
$$
We may write these two as one estimate with $k =1,2:$
\begin{equation}\label{eq:smooth_coutour}
 | \overline{f}(y_1) -    \underline{f}(y_1)|  \lesssim \|\p \Omega   \|_{C^{k,\gamma} }   |  y_1 |^{k+\gamma}.
\end{equation}

%We consider two sub-cases: $ 2\delta < r$ and $2 \delta \geq r $.

%\noindent
%{\it Case 1: $ 2\delta < r$:}

%In this case, we further split the integral in \eqref{aux10151} that bounds $I_1$ as
%\begin{align}
%\big| I_1 \big| \leq  \int_{     0\leq  y_1  \leq 2\delta } \int_{\underline{f}(y_1) \leq y_2 \leq \overline{f}(y_1)  }   +   \int_{     2\delta \leq  y_1  \leq r }  \int_{\underline{f}(y_1) \leq y_2 \leq \overline{f}(y_1)  } := I_{11} + I_{12}
%\end{align}

%For $I_{11}$, we apply the following direct bound:
%\begin{align}\label{eq:all_I_11}
%I_{11} \lesssim \int_{     0\leq  y_1  \leq \delta}   \int_{\underline{f}(y_1) \leq y_2 \leq \overline{f}(y_1)| }   \frac{1}{|y|^{1 + 2\alpha }}    \, dy
%\end{align}
%It follows from \eqref{eq:smooth_coutour} that
%$$
%I_{11} \lesssim   \|\p \Omega   \|_{C^{k,\gamma} }   \int_{     0\leq  y_1  \leq \delta}     y_1^{  k+\gamma -2\alpha-1 }     \, dy_1.
%$$
%Since $ k+\gamma - 2\alpha -1\geq 0$ and $2\delta < r$, we have
%\begin{align} \label{eq:desired_I_11}
%I_{11} \lesssim  \|\p \Omega   \|_{C^{k,\gamma} }  \delta^{k+\gamma -2\alpha } \lesssim  \|\p \Omega   \|_{C^{k,\gamma} }  \delta r^{ k+\gamma -2\alpha-1 }.
%\end{align}

%For $I_{12}$,
By the Mean Value theorem, for some $z_2$ lying between $y_2$ and $y_2 - \delta$ we have
\begin{equation}\label{eq:refiend_MVT}
\left| \frac{1}{|y|^{2 + 2\alpha }} - \frac{ 1}{|y - (0,\delta)|^{2 + 2\alpha }}  \right| = c_\alpha   \delta   \frac{|z_2|}{ (y_1^2 + z_2^2)^{2+\alpha}} \leq \frac{c_\alpha \delta  (\delta+ |y_2|)}{ |y_1|^{4 +2\alpha}}  .
\end{equation}
% To get a more refined bound in \eqref{eq:refiend_MVT}, we notice that for all $y_1 \leq r$, we have $y_2 \lesssim y_1$ due to \eqref{rad} and \eqref{eq:smooth_coutour}. So, for \eqref{eq:refiend_MVT}, we have
%\begin{equation}\label{eq:refiend_MVT_2}
%\left| \frac{1}{|y|^{2 + 2\alpha }} - \frac{ 1}{|y - (0,\delta)|^{2 + 2\alpha }}  \right| \lesssim
% \delta   \frac{| y_2 | + \delta }{  | y_1 |^{4+2\alpha}   }
%\end{equation}

To bound $I_1$, we consider two separate integral regions in \eqref{aux10151} where $ \delta \leq y_1 \leq r$ and where $ y_1 \leq \min\{\delta,r\}$. A direct computation using \eqref{aux10151} and \eqref{eq:refiend_MVT} shows that
\begin{align*}
I_{1} & \lesssim  \int_{  y_1   \leq  \min\{\delta,r\}} \int_{\underline{f}(y_1) \leq y_2 \leq \overline{f}(y_1)  } \frac{dy}{   y_1^{1+2\alpha}   }    +
\delta \int_{     \delta \leq  y_1  \leq r }  \int_{\underline{f}(y_1) \leq y_2 \leq \overline{f}(y_1)  } \frac{(\delta+ |y_2|)dy}{y_1^{3+2\alpha}}
%\\ & \lesssim \delta^2 \|\p \Omega   \|_{C^{k,\gamma}}  \int_{2\delta \leq  y_1  \leq r } y_1^{k+\gamma-3 -2\alpha}    +
.
\end{align*}
Here in the first integral we use a simple direct estimate on the integrand in \eqref{aux10151}. 
Note that the set $ \delta \leq y_1 \leq r$ may be empty, in which case we do not need to bound the latter integral. Then, applying \eqref{eq:smooth_coutour}, we have
\begin{equation}\label{eq:I_1_final}
\begin{aligned}
 I_1   &\lesssim \|\p \Omega   \|_{C^{k,\gamma} }   \int_{     0\leq  y_1  \leq  \min\{\delta,r\} }     y_1^{  k+\gamma -2\alpha-1 }     \, dy_1 +\|\p \Omega   \|_{C^{k,\gamma} } \delta^2 \int_{     \delta \leq  y_1  \leq r } y_1^{k+\gamma-3 -2\alpha}\, dy_1\\
&\quad +  \delta \int_{     \delta \leq  y_1  \leq r } \int_{\underline{f}(y_1) \leq y_2 \leq \overline{f}(y_1)  } |y_2|y_1^{ -3 -2\alpha}\, dy  .
\end{aligned}
\end{equation}

%Using \eqref{eq:refiend_MVT} and \eqref{eq:smooth_coutour}, a direct computation shows that
%\begin{align*}
%I_{1} & \lesssim
%\lesssim  \delta^2 \int_{2\delta \leq  y_1  \leq r } \int_{\underline{f}(y_1) \leq y_2 \leq \overline{f}(y_1)  } \frac{1}{   y_1^{3+2\alpha}   }   +
% \delta \int_{    0 \leq  y_1  \leq r }  \int_{\underline{f}(y_1) \leq y_2 \leq \overline{f}(y_1)  } \frac{1}{y_1^{2+2\alpha}} \lesssim
% \\ & \lesssim \delta^2 \|\p \Omega   \|_{C^{k,\gamma}}  \int_{2\delta \leq  y_1  \leq r } y_1^{k+\gamma-3 -2\alpha}    +
%\|\p \Omega   \|_{C^{k,\gamma} } \delta \int_{     2\delta \leq  y_1  \leq r } y_1^{k+\gamma-2 -2\alpha}.
%\end{align*}

Since $3 \geq   k+\gamma \geq  1+2\alpha,$ the first two integrals in \eqref{eq:I_1_final} are bounded by $\|\p \Omega   \|_{C^{k,\gamma} } \delta r^{k+\gamma-1 -2\alpha}$. For the last integral in \eqref{eq:I_1_final}, we use \eqref{eq:smooth_coutour} together with a non-optimal bound (only optimal when $k=1$)
\begin{equation}\label{eq:smooth_coutour2}
 | \overline{f}(y_1) +    \underline{f}(y_1)|  \leq \|\p \Omega   \|_{C^{k,\gamma} }   |  y_1 |^{1+\gamma}.
\end{equation}
to obtain that
\begin{align*}
\int_{     \delta \leq  y_1  \leq r } \int_{\underline{f}(y_1) \leq y_2 \leq \overline{f}(y_1)  } |y_2|y_1^{ -3 -2\alpha}\, dy &\lesssim  \int_{     \delta \leq  y_1  \leq r }  \Big|  \underline{f}(y_1) -   \overline{f} (y_1)    \Big| \Big|  \underline{f}(y_1) +   \overline{f} (y_1)    \Big|  y_1^{ -3 -2\alpha}\, dy_1 \\
&\lesssim \|\p \Omega   \|_{C^{k,\gamma} }^2 \int_{     \delta \leq  y_1  \leq r }     y_1^{ k+2 \gamma-2 -2\alpha}\, dy .
\end{align*}
Since $k+ 2\gamma -2 -2\alpha \geq \gamma -1 >-1$, combining the terms in \eqref{eq:I_1_final} we obtain
\begin{align}
I_{1}     &\lesssim %\|\p \Omega   \|_{C^{k,\gamma} } \delta^{k+\gamma -2\alpha}      +
\|\p \Omega   \|_{C^{k,\gamma} } \delta r^{ k+\gamma -2\alpha-1} \big( 1+ \|\p \Omega   \|_{C^{k,\gamma} }      r^{ \gamma} \big).    \label{I1fin}
%&\lesssim \|\p \Omega   \|_{C^{k,\gamma} } \delta (r^{ k+\gamma -2\alpha-1}      +       r^{2\alpha})  .
\end{align}
%Therefore, we can conclude from \eqref{eq:desired_I_11} and \eqref{eq:desired_I_12} that $I_1$ satisfies the desired estimate:
%\begin{align}\label{I1fin}
%I_1 \lesssim
%\|\p \Omega   \|_{C^{k,\gamma} }   \delta r^{k+ \gamma -2\alpha-1 }
%\end{align}

%\noindent
%{\it Case 2: $ 2\delta \geq  r$:}

%Here we bound $I_1 $ similarly to $I_{11}$ of the previous case:
%\begin{align}\label{eq:case2:desired_I_1}
%I_{1} & \lesssim  \int_{     0\leq  y_1  \leq r_0}   \int_{\underline{f}(y_1) \leq y_2 \leq \overline{f}(y_1)| }   \frac{1}{|y|^{1 + 2\alpha }}    \, dy
%\end{align}
%It follows from \eqref{eq:smooth_coutour} again that
%\begin{align*}
%I_{1} \lesssim     \|\p \Omega   \|_{C^{k,\gamma} }   \int_{  0 \leq  y_1 \leq r}   |y_1|^{k+\gamma -1 - 2\alpha  }   \, dy_1 .
%\end{align*}
%Since $ k+\gamma  \geq   1+ 2 \alpha$ and $ 2\delta \geq  r$, we obtain
%\begin{align*}
%I_{1} \lesssim     \|\p \Omega   \|_{C^{k,\gamma} }   r^{\gamma - 2\alpha+ k } \lesssim    \|\p \Omega   \|_{C^{1+k,\gamma} }  \delta r^{k+ \gamma - 2\alpha-1 }.
%\end{align*}

%Combining the two cases of $ 2\delta \leq  r$ and $ 2\delta \geq  r$, and using $r <1,$ we have
%\begin{align*}
%I_1 \lesssim
%\|\p \Omega   \|_{C^{k,\gamma} }   \delta r^{k+ \gamma-1 -2\alpha }&

%\end{align*}
\noindent
{\bf Estimates of $I_2$}

Similarly to the estimation of $I_1$, for any $y  \geq 0$ we introduce
$$
\overline{g}(y):= \max\{g(y), g(-y)) \} \quad \text{and }\quad \underline{g}(y):= \min\{g(y), g(-y)) \}.
$$
to obtain
\begin{align*}
 |I_2 | \leq  \int_{     0\leq  y_1  \leq r}\int_{ \underline{g}(y_1) \leq y_2 \leq \overline{g} (y_1)| }  y_1 \left| \frac{1}{|y|^{2 + 2\alpha }} - \frac{ 1}{|y - (0,\delta)|^{2 + 2\alpha }}  \right|  \, dy.
\end{align*}

Since $g(0)=\delta$ and  $g' (0)= 0$, by the same reasoning for \eqref{eq:smooth_coutour}, we have for $k=1,2,$
$$
| \overline{g}(y_1) -    \underline{g}(y_1)| \leq\|\p \Omega   \|_{C^{k,\gamma} }   |  y_1 |^{k+\gamma},
$$
and
$$
| \overline{g}(y_1) +    \underline{g}(y_1)| \leq 2\delta+ \|\p \Omega   \|_{C^{k,\gamma} }   |  y_1 |^{1+\gamma}.
$$
A similar argument shows that $I_2$ shares the same estimate as $I_1$:
\begin{align}\label{I2fin}
I_2   \lesssim
\|\p \Omega   \|_{C^{k,\gamma} }   \delta   r^{ k+\gamma -2\alpha-1} \big( 1+ \|\p \Omega   \|_{C^{k,\gamma} }      r^{ \gamma} \big).
\end{align}

\noindent
{\bf Estimates of $I_3$:}

For the last term $I_3$, we consider two cases:  $2\delta < r$ and  $2\delta \geq  r$ and will show that in both cases
$$
I_3    \lesssim \delta      r^{- 2\alpha }  .
$$

\noindent
{\it Case 1: $ 2\delta \leq   r$:}

By the Mean Value theorem,
\begin{equation}\label{eq:mean_value22}
\left| \frac{1}{|y|^{2 + 2\alpha }} - \frac{ 1}{|y - (0,\delta)|^{2 + 2\alpha }}  \right| \leq c_\alpha  \delta \max_{z_2 \in [y_2-\delta,y_2]}\frac{|z_2|}{ (y_1^2 + z_2^2)^{2+\alpha}}. %\leq c_\alpha  \delta \frac{1}{ |y  |^{3+2\alpha}}
\end{equation}
For $y \in (\RR^2\setminus E) \cap {\rm supp}(\omega)$ we have, due to the definition of admissible points, that $|y| \geq r.$ Since $\delta \leq r/2,$ it follows that
\begin{equation}\label{aux10152} \max_{z_2 \in [y_2-\delta,y_2]}\frac{|z_2|}{ (y_1^2 + z_2^2)^{2+\alpha}} \leq \max_{z_2 \in [y_2-\delta,y_2]} \frac{1}{ (y_1^2 + z_2^2)^{\frac{3}{2}+\alpha}} \leq 2^{3+2\alpha} |y|^{-3-2\alpha}. \end{equation}
It follows from \eqref{eq:mean_value22}, \eqref{aux10152} and the definition of $I_3$ that
%\begin{align*}
%I_3 & \lesssim  \int_{\RR^2\setminus E } \frac{|\o(y)|}{ |y  |^{2+2\alpha}}  \, dy .
%\end{align*}
%Notice that by definition of the set $ E $ and the fact that $\alpha>0$,
\begin{align}
I_3 &  \lesssim  \delta   \int_{   |y|\geq  r } \frac{|\o(y)|}{ |y  |^{2+2\alpha}}\,dy  \lesssim \delta      r^{- 2\alpha }.
\end{align}

\noindent
{\it Case 2: $ 2\delta \geq   r$:}

In this case, we split the integral:
\begin{align*}
I_3 & =  \int_{ \widetilde{E}^c  } y_1 \left( \frac{1}{|y|^{2 + 2\alpha }} - \frac{ 1}{|y - (0,\delta)|^{2 + 2\alpha }}  \right) \o(y) \, dy  + \\ & \int_{|y - (0, \delta/2)|> 4\delta   } y_1 \left( \frac{1}{|y|^{2 + 2\alpha }} - \frac{ 1}{|y - (0,\delta)|^{2 + 2\alpha }}  \right) \o(y) \, dy
 := I_{31} + I_{32}
\end{align*}
where $\widetilde{E}^c$ is defined by
$$
\widetilde{E}^c : = \{y \in \Omega (t) : |y - (0, \delta/2)|\leq 4\delta    \}  \cap E^c,
$$
with $(0, \frac{\delta}{2})$ being the mid-point between $p$ and $q$.

For $I_{31}$,  we use the triangle inequality and the bound
$$
\left| \frac{y_1}{|y|^{2 + 2\alpha }}\right|   + \left| \frac{y_1}{|y - (0,\delta)|^{2 + 2\alpha }}  \right| \leq    \frac{1}{|y|^{1 + 2\alpha } } + \frac{ 1}{|y - (0,\delta)|^{1 + 2\alpha }}
$$
to obtain
\begin{align*}
I_{31} & \lesssim  \int_{  \widetilde{E}^c  }  \frac{1}{|y|^{1+2\alpha }}  \, dy  +  \int_{  \widetilde{E}^c  }   \frac{ 1}{|y - (0,\delta)|^{1 + 2\alpha }}   \, dy \\
&\lesssim   \int_{  r \leq  |y|\leq 5\delta   }  \frac{1}{|y|^{1+2\alpha }}  \, dy \lesssim \delta^{1-2\alpha} \lesssim \delta r^{-2\alpha}.
\end{align*}
%Since $2\delta \geq  r_0$, taking the maximum of the integrand gives
%$$
%I_{31}  \lesssim    \delta r_0^{-2\alpha}.
%$$

For $I_{32}$, similarly to \eqref{eq:mean_value22} and \eqref{aux10152}, for $y \in \RR^2\setminus \widetilde{E}^c $ we have
\begin{equation*}
\left| \frac{1}{|y|^{2 + 2\alpha }} - \frac{ 1}{|y - (0,\delta)|^{2 + 2\alpha }}  \right| \leq c_\alpha  \delta \max_{ z_2 }\frac{z_2}{ (y_1^2 + z_2^2)^{2+\alpha}} \lesssim   \delta \frac{1}{ |y  |^{3+2\alpha}}.
\end{equation*}

Inserting this bound into the integral $I_{32}$ gives
\begin{align*}
I_{32} & \lesssim  \delta   \int_{   |y|\geq 2\delta  } \frac{1}{ |y  |^{2+2\alpha}}  \, dy   \\
& \lesssim \delta^{ 1-  2\alpha }  \lesssim   \delta   r^{- 2\alpha }.
\end{align*}

Putting together $I_{31}$ and $I_{32}$ yields
\begin{equation*}
I_3 \lesssim \delta      r^{- 2\alpha }.
\end{equation*}

\noindent
{\bf Combined estimates:}

Now observe that with our choice of $r$ \eqref{rad}, $r^{k+\gamma  -1 } \leq \|\p \Omega\|_{C^{k,\gamma}}^{-1},$ and combining the estimates of $I_1$, $I_2$, and $I_3$, we arrive at
\eqref{udiffeq}:
\begin{align*}
\big| u_n(q) - u_n(p)\big|   & \lesssim \delta \Big( r^{-2\alpha}  + r^{ k+\gamma -2\alpha-1} \|\p \Omega   \|_{C^{k,\gamma} } \big(1 + \|\p \Omega   \|_{C^{k,\gamma} }  r^{\gamma}\big)    \Big)  \\
 & \lesssim \delta \Big( r^{-2\alpha}  + r^{-2\alpha} \big(1 + \|\p \Omega   \|_{C^{k,\gamma} }  r^{\gamma}\big)    \Big)  \\
& \lesssim \delta  \|\p \Omega   \|_{C^{k,\gamma} }^{\frac{2\alpha }{ k+\gamma  -1 } } \big( 1+    \|\p \Omega   \|_{C^{k,\gamma} }  r^{\gamma}    \big) \lesssim   \|\p \Omega \|_{C^{k,\gamma}}^{\frac{k+ 2\alpha -1   }{ k+\gamma  -1 }}|p-q|.
\end{align*}
Here the constants in the last two inequalities may depend on $\rho.$ 
\end{proof}

\subsection{Estimates on the distance and proof of main results}

Now fix $1>\alpha>0$ and suppose that we have a $C^{k,\gamma}$ patch solution $\Omega(t) = \cup_{j=1}^N \Omega_j(t)$ with $k = \lceil 2 \alpha \rceil,$ and $1 \geq \gamma  \geq   2 \alpha  - \lceil 2 \alpha \rceil+1$
that forms a splash singularity described in Definition~\ref{defss} at some point $x_0$ at time $T$ (it is not hard to generalize the argument to the case where
splash happens simultaneously at a finite number of different points).

For any $i\neq j$, let us define $d_{ij}(t) = {\rm dist}(\Omega_i(t), \Omega_j(t)).$ To control patch self-intersections, fix a small parameter $\eta>0$, and define
$$
d_{ii}(t) = {\rm min}_{s,r: |s-r| \geq \eta} |x_i(s)-x_i(r)|,
$$
where $x_i(s)$ is the arc length parametrization of $\Omega_i(t).$
Given any splash singularity, one can choose sufficiently small $\eta>0$ so that we must have $d_{ii}(t) \rightarrow 0$ as $t \rightarrow T;$ indeed, it suffices to choose $\eta = \rho,$ the radius of the ball from Definition~\ref{defss}.

Finally, define the  minimal distance $d(t)$ between different patches (or different branches of the same patch if itself-intersects) by
\begin{equation}\label{ddef}
d(t) = {\rm min}_{1 \leq i,j \leq n} d_{ij}(t).
\end{equation}
Due to our definition of splash singularity, there exists $\ep>0$ such that for every $T-\ep \leq t < T,$ all points
such that $d(t) = |p-q|$ must lie in $B_{\rho/4}(x_0)$. Note that due to our definition of patch solution and boundedness of the normal component of the velocity $u_n$ at $\p \Omega$
for our patch regularity assured by Lemma~\ref{uinf}, all functions $d_{ik}(t)$ are Lipschitz in time, and therefore $d(t)$ retains this property.

Now we are ready to prove
\begin{proposition}\label{dist1015p}
Let $0 < \alpha < 1$, $k :=  \lceil 2 \alpha \rceil,$ and $1 \geq \gamma  \geq   2 \alpha  - \lceil 2 \alpha \rceil+1.$
Suppose that $\o$ is a $C^{k,\gamma}$-patch solution that forms a
splash singularity as in Definition \ref{defss} at time $T$. Then the minimal distance $d(t)$ defined in \eqref{ddef} is a Lipschitz function of time and there exists $\ep>0$ such that for almost every time $t \in [T-\ep,T)$
\begin{equation}\label{dist1015e}
d'(t) \geq - C d(t) \|\p \Omega \|_{C^{k,\gamma}}^{\frac{k+ 2\alpha -1    }{k+\gamma-1}}.
\end{equation}
\end{proposition}
\begin{proof}
Since $d(t)$ is Lipschitz, it is almost everywhere differentiable by the Rademacher theorem, and moreover $d(t_2)-d(t_1) = \int_{t_1}^{t_2} d'(t)\,dt$ (see e.g. \cite{EG}).
Fix a time $t \in [T-\ep,T);$ according to our definition of splash singularity,
$\p \Omega \cap B_\rho(x_0)$ consists of two disjoint simple curves that we will denote $\cC_1 (t)$ and $\cC_2 (t)$ and we can assume that $d(t) = {\rm dist}(\cC_1 (t) \cap B_{\rho/2}(x_0),
\cC_2 (t)\cap B_{\rho/2}(x_0)).$

 Due to our definition of patch solution, we have that for small $h>0$
\begin{equation}\label{distc1c2}
d(t+h) = {\rm dist}(\cC_1(t+h), \cC_2(t+h)) = {\rm dist} (X_{u_n(\cdot, t)}^h(\cC_1(t)),X_{u_n(\cdot, t)}^h(\cC_2 (t))) + o(h),
\end{equation}
where $o(h)$ means a quantity that goes to zero as $h \to 0^+$.

Our aim will be to derive a lower bound on $d(t+h)$ using \eqref{distc1c2}. Let $S \in \cC_1 (t) \times \cC_2 (t)$ be the set of all pairs of points $(p,q) \in \cC_1 (t) \times \cC_2 (t)$ such that $d(t)=|p-q|.$ Due to our definition of admissible points and \eqref{rad},
any pair $p,q$ is admissible and so the estimates of Proposition \ref{prop:main_alpha_all} apply.

To bound the first term on the right-hand side of \eqref{distc1c2}, let us fix a small number $\ep_1>0.$ Define a distance $\zeta$ on $ \cC_1 (t) \times \cC_2 (t)$ by \[
\zeta((x,y),(p,q)) = |x-p|+|y-q| \quad \text{for $(x,y), (q,p) \in \cC_1 \times \cC_2$.}
\]
Let us denote $S_{\ep_1}$ the set of pairs of points $(x,y)$ such that
$\zeta((x,y),S) < \ep_1.$  This admits a decomposition of $ \cC_1 (t) \times \cC_2 (t) $:
\begin{equation}
 \cC_1 (t) \times \cC_2 (t)=  S_{\ep_1} \cup S_{\ep_1}^c
\end{equation}
where the compact set $S_{\ep_1}^c := \cC_1 (t) \times \cC_2 (t) \setminus S_{\ep_1}$ consists of pairs $(x,y)$ that are away from the admissible ones.

Suppose first that $(x,y) \in S_{\ep_1}^c,$ that is, $\zeta((x,y), S) \geq \ep_1>0.$
Then there exists $\eta(\ep_1)>0$ such that
$|x-y| \geq d(t)+\eta(\ep_1).$ Indeed, $|x-y| - d(t)$ is a continuous function on the compact set $S_{\ep_1}^c$, and so it has a minimum that is clearly positive.
%Indeed, if not, we could find a sequence $(x_n,y_n)$ with $\zeta((x_n,y_n), S) \geq \ep_1$ such that $|x_n - y_n| \rightarrow d(t).$ Passing to a converging subsequence,
%we find that its limit $(\tilde x, \tilde y) \in S$ while also $\zeta((\tilde x, \tilde y), S) \geq \ep_1,$ a contradiction.
Therefore, for all $(x,y) \in S_{\ep_1}^c$, we have
\begin{equation}\label{eq:zetacase_1}
 |x+u_n(x)h-y - u_n(y)h| \geq |x-y| - 2\|u_n\|_{C(\p \Omega)} h \geq d(t)+\eta(\ep_1) -C(\Omega)h \geq d(t)
\end{equation}
for all $h>0$ sufficiently smaller than $\eta(\ep_1) $. Thus, the points that are not in $S_{\ep_1}$ are not important in derivation of the lower bound on $d'(t).$

Now consider $(x,y) \in S_{\ep_1}.$ Find $(p,q) \in S$ such that $\zeta((x,y),(p,q)) \leq \ep_1.$
Note that by Lemma \ref{uinf} and compactness of $\p \Omega,$ $u_n$ is uniformly continuous on $\p \Omega.$ Therefore,
since $|x-p| \leq \ep_1$ and $|y-q| \leq \ep_1,$ there exists $\sigma(\ep_1)$ such that $|u_n(x)-u_n(p)| \leq \sigma$
and  $|u_n(y)-u_n(q)| \leq \sigma,$ and $\sigma(\ep_1) \rightarrow 0$ as $\ep_1 \rightarrow 0.$
Applying these considerations along with the bound \eqref{udiffeq}, we find that for $(x,y) \in S_{\ep_1}$, the following bound holds:
\begin{align}
|x+u_n(x)h-y - u_n(y)h| &\geq |x-y| - h|u_n(x)-u_n(y)| \nonumber \\
&\geq  |p-q| - \zeta((x,y),(p,q))  \nonumber\\
&\qquad - h|u_n(x)-u_n(p)| -  h|u_n(p)-u_n(q)|- h|u_n(q)-u_n(y)| \nonumber\\
&\geq
d(t) - \ep_1 - 2h \sigma(\ep_1) - C h d(t) \|\p \Omega \|_{C^{k,\gamma}}^{\frac{k+ 2\alpha -1    }{k + \gamma  -1 }} \label{eq:zetacase_2}.
\end{align}

It follows from \eqref{distc1c2}, \eqref{eq:zetacase_1}, and \eqref{eq:zetacase_2} that given any $\ep_1>0$, for any sufficiently small $h>0$, we have
\[ d(t+h) \geq d(t) - \ep_1 -2h \sigma(\ep_1) - C   h d(t) \|\p \Omega \|_{C^{k,\gamma}}^{\frac{k+ 2\alpha -1    }{k + \gamma  -1}} -E(h), \]
where $E(h)=o(h).$
Since this inequality holds for every $\ep_1>0 $, \eqref{dist1015e} follows for a.e. $t  \in [T-\ep,T).$
\end{proof}

%\subsection{Proof of main results}

\begin{proof}[Proof of Theorems \ref{thm:main_thm_no_splash} and \ref{thm:main_thm_exp_bound}]
We only proof Theorem \ref{thm:main_thm_no_splash}, since given the Assumption \ref{as1}, the proof of Theorem \ref{thm:main_thm_exp_bound} follows along the same lines.

According to Proposition \ref{dist1015p}, for a.e. $t \in [T-\ep,T)$ before the splash singularity \eqref{dist1015e} holds.
By Gronwall lemma, we must have that
\[ \int_0^T \|\p \Omega(t) \|_{C^{k,\gamma}}^{\frac{k+ 2\alpha -1    }{  k+\gamma-1 }}\,dt = \infty, \]
completing the proof of Theorem \ref{thm:main_thm_no_splash}.

\end{proof}

{\bf Acknowledgement.} \rm AK acknowledges partial support
of the NSF-DMS grant 2006372 and of the Simons Fellowship grant 667842. XL is partially supported by the NSF-DMS grant 1926686.
We thank Andrej Zlatos for noticing a miscalculation in the earlier version of the paper.

\end{document}